\documentclass[a4paper,reqno,10pt]{amsart}
\usepackage[a4paper,left=3cm,right=3cm,top=3cm,bottom=3.5cm]{geometry}
\usepackage{marginnote}

\usepackage{booktabs}

\usepackage[english]{babel}
\usepackage[T1]{fontenc}

\usepackage{amssymb,amsfonts,amsmath,amsthm}

\usepackage{mathrsfs}
\usepackage{stmaryrd}

\usepackage{stackengine}
\usepackage{mathtools}

\usepackage[hidelinks,unicode=true]{hyperref}
\usepackage{color}

\usepackage{dsfont}
\usepackage{comment}
\usepackage{lastpage}
\usepackage{appendix}
\usepackage{esvect}
\usepackage{esint}

\usepackage{float}

\usepackage{tikz}
\usetikzlibrary{decorations.pathmorphing,patterns}
\usepackage{tkz-graph}
\usepackage{xcolor}
\usepackage{pgfplots}
\pgfplotsset{width=7cm,compat=1.5.1}

\usetikzlibrary{arrows.meta,calc,decorations.pathreplacing,intersections}
\usetikzlibrary{decorations.markings, positioning}

\renewcommand{\paragraph}[1]{$\empty$\\\textit{#1}.\hspace{0.5em}}

\numberwithin{equation}{section}

\usepackage{enumitem}

\newenvironment{itempuce}
{\begin{itemize}[
	label=$-$, 
	leftmargin=*, 
]}{\end{itemize}}

\newenvironment{itemnum}[1]{
\begin{enumerate}[
	label=$(#1)$, 
	leftmargin=*,
	parsep=0cm, itemsep=0cm,
	topsep=0cm]}
{\end{enumerate}}

\newcommand{\rom}{\roman*}

\newtheorem{thm}{Theorem}[section]

\newtheorem{lem}[thm]{Lemma}

\newtheorem{cor}[thm]{Corollary}

\newcommand{\fa}{\forall \,}

\renewcommand{\and}{\qquad\mathrm{and}\qquad}

\renewcommand{\leq}{\,\leqslant\,}
\renewcommand{\geq}{\,\geqslant\,}
\newcommand{\<}{\,<\,}
\renewcommand{\>}{\,>\,}
\renewcommand{\=}{\,=\,}
\newcommand{\Llra}{\Longleftrightarrow}

\newcommand{\lra}{\longrightarrow}

\newcommand{\xra}[2]{\;\;\xrightarrow[#2]{#1}\;\;}
\newcommand{\lms}{\longmapsto}

\newcommand{\ps}[1]{\left\langle #1 \right\rangle}

\newcommand{\event}[1]{\left\{#1\right\}}
\newcommand{\vabs}[1]{\left|#1\right|}
\newcommand{\parent}[1]{\left(#1\right)}
\newcommand{\intervalle}[1]{\left[#1\right]}
\newcommand{\norme}[2]{\left\|#1\right\|_{#2}}

\newcommand{\dist}{\dist}

\newcommand{\mtext}[1]{\mbox{\rm #1}}
\newcommand{\textm}[1]{\;\; \mbox{\rm #1}}

\newenvironment{acc}{\left\{\begin{tabular}{ll}}{\end{tabular}\right.}

\renewcommand{\d}{\mathrm{d}}

\newcommand{\Hess}{\mathrm{Hess}}

\newcommand{\C}{\mathrm{C}}
\newcommand{\D}{\mathrm{D}}
\newcommand{\E}{\mathrm{E}}

\newcommand{\G}{\mathrm{G}}

\newcommand{\J}{\mathrm{J}}

\renewcommand{\L}{\mathrm{L}}

\renewcommand{\O}{\mathrm{O}}
\renewcommand{\P}{\mathrm{P}}

\renewcommand{\S}{\mathrm{S}}
\newcommand{\T}{\mathrm{T}}
\newcommand{\U}{\mathrm{U}}
\newcommand{\V}{\mathrm{V}}

\newcommand{\1}{\mathds{1}}

\newcommand{\EE}{\mathbb{E}}

\newcommand{\RR}{\mathbb{R}}

\newcommand{\tmix}{t_{\text{mix}}}

\newcommand{\Law}{\mtext{Law}}
\renewcommand{\dist}{\mtext{dist}}

\newcommand{\Kullback}{\mtext{Kullback}}

\newcommand{\Wasserstein}{\mtext{Wasserstein}}

\newcommand{\Ric}{\mtext{Ric}}
\newcommand{\Var}{\mtext{Var}}

\title{On the cutoff phenomenon for Dyson--Jacobi processes}
\author{Samuel Chan-ashing}
\address{
DMA, École normale supérieure, Université PSL, CNRS, 75005 Paris, France \newline
CEREMADE, Université Paris-Dauphine, PSL, CNRS \newline
CMAP, Inria, CNRS, École polytechnique, Institut Polytechnique de Paris, 91120 Palaiseau, France}
\email{samuel.chan-ashing@ens.psl.eu}

\date{\today}

\keywords{Dyson--Jacobi processes; 
	Markov diffusion processes;
	Interacting Particle System,
	Random Matrix Theory;
	Cutoff phenomenon;
	Spectral analysis;
	Intrinsic Wasserstein distance;
	Curvature-dimension inequality;
	}

\subjclass[2000]{60J60 (Diffusion processes); 82C22 (Interacting particle systems)}

\begin{document}

\maketitle

\begin{abstract}
We study the convergence to equilibrium of the Dyson--Jacobi process, a system of $n$ interacting particles on the segment $[0,1]$ arising from Random Matrix Theory. We establish the occurence of a cutoff phenomenon for the intrinsic Wasserstein distance and provide an explicit formula for the associated mixing time.

Our approach relies on the interplay between the Riemannian geometry of the process and a flattened Euclidean representation obtained via a diffeomorphic deformation. This transformation allows us to transfer curvature-dimension inequalities from the Euclidean setting to the original space, thereby yielding sharp quantitative estimates.
\end{abstract}

\tableofcontents


\section{Introduction and main results}

The \emph{cutoff phenomenon} describes an abrupt transition in the convergence to equilibrium of a parameter-indexed family of ergodic diffusions. It reflects a competition between intrinsic relaxation toward equilibrium and dependence on the initial condition, with the parameter determining the regime in which this competition is resolved. When cutoff occurs, the distance to equilibrium remains close to its maximal value until a sharply defined time, after which it rapidly collapses to near zero.
This notion was introduced by David Aldous and Persi Diaconis in the 1980s in their study of random walks on finite sets; see for instance \cite{AD86, Dia96}. 
Guan-Yu Chen and Laurent Saloff--Coste showed in \cite{CSC08} that the \emph{product condition}, introduced by Yuval Peres, is sufficient to guarantee cutoff in $\L^p$ distance for $p>1$. The case $p=1$, corresponding to total variation distance, also holds for diffusions with positive curvature and was recently established by Justin Salez in \cite{Sal25}. These general methods, however, do not provide sharp estimates of the mixing time. In this direction, Djalil Chafaï and Max Fathi proposed in \cite{CF25} a precise analysis of the cutoff phenomenon for certain overdamped Langevin diffusions with positive curvature in Euclidean space.

This work continues the program of obtaining sharp mixing time estimates for Dyson diffusions, a family of interacting particle systems arising in random matrix theory. We follow the line of investigation initiated for the Dyson--Ornstein--Uhlenbeck process by Jeanne Boursier, Djalil Chafaï and Cyril Labbé in \cite{BCL}, and for the Dyson--Laguerre process in \cite{Cha25}. We focus here on the Dyson--Jacobi processes associated with the $\beta$-Jacobi ensembles. These dynamics were studied notably by Yan Doumerc, Nizar Demni and Ezéchiel Kahn; see \cite{Dou05, Dem10, Kah21}. In the non-interacting case ($\beta = 0$), the cutoff phenomenon and the mixing time were completely resolved as early as 1994 by Saloff--Coste in \cite[Section 4.1]{SC94}.

Our approach combines diffeomorphic deformation results for diffusions, which allow us to work in a Euclidean setting whenever necessary. More precisely, we observe that curvature-dimension inequality, distances and divergences, and related functional inequalities are invariant under such deformations. In particular, in the Dyson--Jacobi case, this strategy allows us to reduce the dynamics to a Langevin-type diffusion for which curvature-dimension conditions can be established, together with the consequences that follow from them. We can apply the results  obtained in \cite[Appendix B]{Cha25}, which yields the cutoff phenomenon and sharp mixing time estimates for the intrinsic Wasserstein distance.

\subsection{The Dyson--Jacobi process}

In what follows, we study the \textit{Dyson--Jacobi process} (DJ), defined as the solution \(X^n\) to the stochastic differential equation (SDE)
\begin{equation}
\label{eq:DJ}
\d X_{t}^{i,n} \= \sqrt{2 X_{t}^{i,n}(1- X_{t}^{i,n})} \, \d B^i_t \;+ \parent{b_n - (a_n+b_n) X_{t}^{i,n}}\d t + \; \frac{\beta}{2} \sum_{j \neq i}^n\frac{ H(X_{t}^{i,n}, X_{t}^{j,n})}{ X_{t}^{i,n}- X_{t}^{j,n}}\d t,
\end{equation}
where $H(x_i,x_j) = x_i(1- x_j)+ x_j(1- x_i)$ and the parameters satisfy $n\geq 1$, $\beta\>0$ and $a_n, b_n\geq0$.

\paragraph{Existence and uniqueness}
The system may exhibit blow up phenomena, either because the diffusion coefficients degenerate at the boundaries, where they may vanish, or due to collisions between particles, which repel each other through a Coulomb-type interaction. To avoid such pathologies, we restrict attention to a parameter regime that guarantees well-posedness of the dynamics.

A necessary condition for the existence of a global solution to \eqref{eq:DJ} is $a_n \wedge b_n \> \frac{2}{\beta} (n-1)$ for $\beta > 0$, see for instance \cite[Remark 2.2]{Kah21}. Under this condition on the parameters $a_n$ and $b_n$, the existence and uniqueness of strong solutions for general $\beta$ are covered by \cite{GM13} for the case $\beta \geq 1$, and completed by \cite{GM14, Dem10, Kah21} for the case $0 < \beta < 1$. In the non-interactive case $\beta=0$, the dynamics reduce to $n$ independent particles following the real Jacobi SDE studied in \cite[Lemma 5.1]{Dou05}.

More precisely, \cite[Corollary 9]{GM13} states that the system \eqref{eq:DJ} admits a unique global strong solution on $[0,+\infty)$ with no particle collisions provided that the ordering condition satisfies the ordering constraint $0 \leq x_0^{1,n} \< \hdots \< x_0^{n,n} \leq 1$ a.s., 
and that
\begin{equation}\label{eq:exun}
\beta \geq 1 \and a_n \wedge b_n \> \frac{2}{\beta} (n-1).
\end{equation}
Although cutoff is expected to hold under broader assumptions, we will adopt this standing condition throughout. Moreover, this parameter regime, together with the ordered initial condition, ensures that the system remains almost surely inside the convex domain of ordered nonnegative coordinates
\begin{equation}
X_t^n \in \D_n \;:=\; \event{x \in \RR^n :\; 0 \leq x_1 \< \cdots \< x_n \leq 1}.
\end{equation}

\paragraph{Notations}
Observe that if one fixes an initial condition $x_0^n \in \D_n$ for each $n \geq 1$, then the sequence given by $(x_0^{n,n})_n$ is strictly increasing and bounded above by $1$. It therefore converges to a strictly positive limit, which we denote throughout by $\bar{x}_0 \in (0,1]$. In particular, we then have the convergence $\frac{1}{n} \sum_{i=1}^{n} x_{0}^{i,n} \xra{}{n \to \infty} \bar{x}_0$. We denote by $\D$ the set of strictly increasing sequences with values in $(0,1)$. For any $x_0 \in \D$, we write $x_0^n$ for the first $n$ terms of this sequence.

\paragraph{Invariant measure}
This structural connection with matrix models suggests viewing the Dyson--Jacobi process as the dynamical analogue of the Jacobi ensemble in random matrix theory.
This link is made through the invariant measure associated to the particle system, which corresponds to the $\beta$-Jacobi ensemble \cite{KN04, For10}. Recall that these $\beta$-ensembles generalize the classical matrix cases $\beta = 1$, $\beta = 2$ and $\beta = 4$ corresponding respectively to the JOE, JUE and JSE of \cite{For10}.
 
More precisely, according to \cite[Proposition 2.6]{Kah21}, the Dyson--Jacobi process \eqref{eq:DJ} admits a unique stationary probability measure $\pi_{\beta}^n$, whose density with respect to the Lebesgue measure is given by the multivariate Beta distribution
\begin{equation}\label{eq:gc}
\d \pi_\beta^n(x) \;:=\; \frac{\1_{{(x_1,\ldots,x_n)\in\overline{\D}_n}} }{C_n^\beta} \prod_{i=1}^n x_i^{{b_n- \frac{\beta}{2} (n-1)}-1} (1-x_i)^{{a_n - \frac{\beta}{2} (n-1)}-1} \prod_{i > j}(x_i-x_j)^{\beta} \, \d x_i,
\end{equation}
where $C_n^\beta$ is a normalizing constant.
The invariant measure is the Gibbs measure associated with the energy, defined for $x\in\D_n$ by
\begin{equation}
\E(x) = -\sum_{i=1}^n \parent{c_b \ln {x_i} + c_a \ln (1 - x_i) + \frac{\beta}{4} \sum_{j < i} \ln\vabs{x_i-x_j}}.
\end{equation}
with $c_b = b_n - \frac{\beta}{2} (n-1)-1$ and $c_a = a_n - \frac{\beta}{2} (n-1)-1$.
The Coulomb gas $\pi_{\beta}^{n}$ is log-concave on its natural domain whenever the energy is convex, which holds under the following additional condition
\begin{equation}\label{eq:exconv}
a_n \wedge b_n \> \frac{\beta}{2} (n-1) + 1.
\end{equation}

\paragraph{The intrinsic Wasserstein distance} \label{sec:wgdl}
The diffusion coefficient of the Dyson--Jacobi process has components $\sigma : x_i \lms \sqrt{2}/\sqrt{x_i(1-x_i)}$, and the corresponding carré du champ operator is given in \eqref{eq:cdc}.
In this context, we consider the diagonal metric $g$ defined by $g_{ii}(x) \= 1/{x_i(1-x_i)}$  and $g_{ij}(x) \= 0$ if $i\neq j$. This corresponds to the diagonal setting of \cite[Appendix B]{Cha25}, with the function $a : x \lms 1/\sqrt{x_i(1-x_i)}$ on $(0,1)$. Since an antiderivative of $a$ is $A : x \lms 2 \arcsin \sqrt{x}$, the associated Riemannian distance for the Dyson--Jacobi process is
\begin{equation}
\fa x,y \in (\RR_+^\ast)^n \;:\quad d_{g}(x,y) \;:=\; 4 \sqrt{\sum_{i=1}^n \parent{\arcsin\sqrt{x_i}-\arcsin\sqrt{y_i}}^2}.
\end{equation}
Moreover, the geodesic $\gamma$ connecting $x$ to $y$ is given by $\gamma_i(t) \= 4 \parent{t \arcsin\sqrt{y_i} + (1-t)\arcsin\sqrt{x_i}}^2$ for $t\in[0,1]$.
The intrinsic Wasserstein distance of order $r\geq0$ is then defined, for probability measures $\mu$ and $\nu$ on the same space, by
\begin{equation}
\Wasserstein_{g,r}(\mu, \nu) \= 4 \parent{\inf_{( X, Y)} \EE\intervalle{\parent{\sum_{i=1}^n \vabs{\arcsin\sqrt{ X_i}-\arcsin\sqrt{ Y_i}}^2}^{r/2}}}^{1/r},
\end{equation}
where the infimum runs over all couplings $(X, Y)$ of $\mu$ and $\nu$. We simply write $\Wasserstein$ in the case $r=2$.

\subsection{Main results}

The mixing time associated with the Wasserstein distance is defined by 
\[\tmix^n(\varepsilon) \;:=\; \inf \event{t \geq 0 \;:\; \Wasserstein \parent{\Law(X_t^n) \mid \pi_{\beta}^{n}} \leq \varepsilon}.\]
Our main result establishes that the cutoff phenomenon occurs for the Dyson--Jacobi particle and provides an explicit expression for the mixing time.

\begin{thm}[Universal cutoff]\label{thm:CDJ}
Let $x_0$ be in $\D$. 
Let ${(X^n_t)}_{t\geq0}$ be the Dyson--Jacobi process \eqref{eq:DJ} started at $x_0^n$.
Then a cutoff phenomenon occurs for $\Wasserstein$ at critical time $c_n$, namely for all $\varepsilon \in (0,1)$,
\[\lim_{n\to\infty} 
\Wasserstein(\Law( X^n_{t_n})\mid \pi_{\beta}^{n}) \=\begin{cases}
	+\infty 
	& \text{if $t_n=(1-\varepsilon)c_n$},\\
	0 & \text{if $t_n=(1+\varepsilon)c_n$},
\end{cases}\]
assuming that the following lower bound tends to infinity as $n\to\infty$,
\[\frac{1}{a_n+b_n} \log \parent{\sqrt{n \frac{a_n+b_n}{b_n}} \vabs{\bar{x}_{0} - \frac{b_n}{a_n + b_n}} } \leq c_n \leq \frac{\log n \vee \log\vabs{\bar{x}_{0} + \frac{b_n}{a_n+b_n}}}{a_n+b_n-\beta(n-1)}.\]
\end{thm}

\noindent One can refine this theorem to cover non-atomic initial conditions.
\begin{cor}\label{cor:CDJ}
Let ${(X^n_t)}_{t\geq0}$ be the Dyson--Jacobi process \eqref{eq:DJ} under an additional assumption on the coefficients
\[\frac{a_n}{b_n} \xra{}{n\to\infty} r \in [0, +\infty].\]
Then a cutoff phenomenon occurs for $\Wasserstein$ at critical time $c_n$, namely for all $\varepsilon \in (0,1)$,
\[\lim_{n\to\infty} \sup_{x_0^n \in \S^n_0}
\Wasserstein(\Law( X^n_{t_n})\mid \pi_{\beta}^{n}) \=\begin{cases}
	+\infty 
	& \text{if $t_n=(1-\varepsilon)c_n$},\\
	0 & \text{if $t_n=(1+\varepsilon)c_n$},
\end{cases}\]
where the set of initial conditions, and the critical time are the following
\begin{itempuce}
	\item for $r=+\infty$: take $0 \< \alpha \leq 1$ if $a_n \sim n^{\alpha} b_n$, or take $\alpha \= 0$ if $a_n \sim \varepsilon_n b_n$ with $\varepsilon_n = o(n)$, \[\S_0^n \= \D_n \and \frac{1+\alpha}{2}\frac{\log n}{\lambda_n} \leq c_n \leq \frac{\log n}{\lambda_n}.\]
	\item for $r=0$: for any $\varepsilon \> 0$, \[\S_0^n \= \event{x_0^n \in \D_n \;:\; \vabs{\bar{x}_0 - 1} \> \varepsilon} \and \frac{\log n}{2 \lambda_n} \leq c_n \leq \frac{\log n}{\lambda_n}.\]
	\item for $0 \< r \< +\infty$: for any $\varepsilon \> 0$ and $m := \frac{1}{r+1}$, if $\lambda_n \sim \alpha \rho_n$ with $\alpha \in (0, \frac{1}{2})$, \[\S_0^n \= \event{x_0^n \in \D_n \;:\; \vabs{\bar{x}_0 - m} \> \varepsilon} \and \frac{\log n}{2 \lambda_n} \leq c_n \leq \frac{\log n}{2\alpha \lambda_n}.\]
\end{itempuce}
In the other cases, the bounds obtained do not provide sufficiently sharp control of the mixing time.
\end{cor}

\paragraph{Cutoff for other distances or divergences}
Let $\mu$ and $\nu$ be two probability measures on the same space. Their total variation distance is
\[\norme{\mu-\nu}{\T\V} := \sup_A\vabs{\mu(A)-\nu(A)} \in [0,1].\]
The relative entropy (Kullback--Leibler divergence), for $\mu \ll \nu$, is defined by
\[\Kullback(\mu \mid \nu) := \int \frac{\d \mu}{\d \nu} \log \parent{\frac{\d \mu}{\d \nu}} \d \nu \= \int \log \parent{\frac{\d \mu}{\d \nu}} \d \mu \in [0,+\infty],\]
with the convention $\Kullback(\nu \mid \mu) = +\infty$ if $\nu$ is not absolutely continuous with respect to $\mu$.
One observes in \cite{Cha25} that the intrinsic Wasserstein structure provides the appropriate geometric framework for deriving the functional inequalities that relate these distances and divergences. 
In this setting, we make use of the inequality from \cite{PS25}, which allows one to obtain an upper bound on the $\L^2$ distance in terms of total variation distance, and thus to compare the $\L^2$ distance and total variation distance.
For the Fisher distance, cutoff can also be established, in the same spirit as in \cite{CF25}. Note, however, that this distance is intrinsically defined with respect to the underlying process, so the cutoff result should be interpreted as holding for an \textit{intrinsic Fisher distance}. Other distances and divergences, such as the Hellinger distance, the \( \L^p \) distances for \( p > 1 \), or the chi-squared (\(\chi^2\)) distance, can be compared either to the total variation (see \cite{BCL}) or to the \( \L^2 \) distance via interpolation.

\paragraph{Matrix cases}
The Dyson--Jacobi process originates from the matrix Jacobi process introduced in \cite{Dou05} for the real case \( \beta = 1 \) and for the complex case \( \beta = 2 \). The real (resp. complex) matrix Jacobi process is defined as the radial component of left corner of an orthogonal (resp. unitary) Brownian motion. More precisely, let $n \geq 1$ and $a_n$, $b_n$ be integers such that $b_n \geq n+1$ and $a_n \geq n+1$. 
Set $m := a_n+b_n$ and let $\Theta$ be an $m\times m$ orthogonal Brownian motion (see, e.g., \cite{Lev17}). Consider the process $(M_t)_{t \geq 0}$ obtained by  taking the submatrix of $\Theta_t$ formed by the first $n$ rows and $b_n$ columns.
The covariance matrix $M_t^\dag M_t$, where $\dag$ denotes the conjugate transpose, takes values in the cone of positive semidefinite matrices.
A remarkable result from \cite[Theorem 1]{Dou05} states that the eigenvalues of $ M_t^\dag M_t$ evolves as a Dyson--Jacobi particle system \eqref{eq:DJ} with parameter $\beta=1$.
Thus, the Dyson--Jacobi SDEs provide a natural extension of these matrix-derived eigenvalue diffusions, simultaneously encompassing the real $(\beta=1)$ and complex $(\beta=2)$ settings.

\paragraph{Links with the other beta-ensembles}
The Jacobi ensemble occupies a central position in the theory of beta-ensembles. Indeed, the Beta distribution is the most general case in this family, from which the Gamma and Normal distributions arise simply as limiting forms obtained through translation and dilation.


From a dynamical perspective, as discussed in \cite[Section 2.7.4]{BGL14}, Jacobi processes interpolate between the classical beta-ensembles through appropriate scaling limits. Both the Ornstein--Uhlenbeck and Laguerre generators appear as degenerate limits of Jacobi operators. Regarding the convergence to the Ornstein--Uhlenbeck operator, consider the symmetric case with parameters $a_n = b_n = n/2$. If one translates the support $[0,1]$ to the symmetric interval $[-1,1]$ and rescales the spatial variable $x$ by a factor $1/\sqrt{n}$, the rescaled Jacobi operator $\frac{1}{n} \G^{\D\J}$ converges to the Ornstein--Uhlenbeck operator $\G^{\O\U}$ as $n \to \infty$. This convergence holds at the level of invariant measures, eigenvalues, eigenfunctions, and the associated Markov semigroups.

Conversely, to recover the Laguerre operator, one may zoom in on the edge of the spectrum. By mapping $[0,1]$ to $[0,2]$, rescaling the variable $x$ to $x/n$, and sending $a_n \to \infty$ (typically setting $a_n = n$) while keeping $b_n$ fixed, the operator $\frac{1}{n} \G^{\D\J}$ converges to the Laguerre operator $\G^{\D\L}$. This limit reflects the classical convergence of Beta distributions towards Gamma distributions.

Historically, in the context of orthogonal polynomials, these asymptotic relationships were already noted by Ervin Feldheim \cite{Fel42}. From a Markovian perspective, Bakry established that the Laguerre and Hermite (or Ornstein--Uhlenbeck) families appear as properly renormalized limits of the Jacobi semigroups. This hierarchy identifies the Jacobi family as the generic model for such univariate diffusion operators.

\paragraph{Organization of the paper}
The Dyson--Jacobi process admits multiple interpretations, and throughout this paper, we adopt the two complementary viewpoints, switching between the Riemannian and the Euclidean formulations depending on which is more convenient.
We do not rely exclusively on one framework or the other; instead, we exploit both perspectives jointly.
More precisely, the key idea is that many structural properties of the process can be transferred via a diffeomorphic transformation of the particle system. The Euclidean viewpoint arises from such a diffeomorphism, under which the process becomes a Langevin diffusion with a constant diffusion coefficient. We develop this perspective in Section \ref{sec:ddd}, and in particular in Lemma \ref{lem:ddd}, which shows that several structural properties, such as curvature-dimension inequalities and spectral properties, are preserved under this transformation. This strategy does not apply universally: for instance, log-concavity is not generally preserved under arbitrary changes of variables.
However, in the present setting, the transformation is well-behaved, and log-concavity is indeed preserved. This yields a Euclidean framework in which the curvature-dimension inequality can be established, and within which cutoff is obtained uniformly across a wide class of distances. As a consequence, we derive an explicit mixing time in Section \ref{sec:proofs}.

\subsection{Discussions with related works about the cutoff}\label{ss:krc}

In this subsection, we collect results from the literature that allow us to infer the existence of a cutoff phenomenon for the Dyson--Jacobi process.

\paragraph{Cutoff $\L^p$ with $p>1$}
From the spectral analysis carried out in Section \ref{ss:updj}, one can deduce cutoff in $\L^p$ distances with $p>1$ using \cite[Corollary 3.4]{CSC08}. More precisely, to ensure the occurrence of a cutoff in $\L^p$-distance, it suffies that the initial condition $x_0^n$ satisfies
\begin{equation}\label{eq:iclp}
|\varphi_{n}(x_0^{n})|^2 \,\gg\, \frac{n b_n}{b_n + a_n},
\end{equation}
where $\varphi_{n}$ denotes the eigenfunction \eqref{eq:efsp} associated with the eigenvalue $-(a_n+b_n)$ of the generator $\G_{\D\J}$ defined in \eqref{eq:IG}.
Therefore, choosing $x_0^n$ as in \eqref{eq:iclp}, we deduce that the Dyson--Jacobi process exhibits a cutoff in $\L^2(\pi_{\beta}^{n})$ norm for the class of initial conditions \(\S_0^n = \{\delta_{x_0^n}\}\). We emphasize that this argument does not yield an explicit expression for the mixing time, which remains the main objective of the present work.

\paragraph{Cutoff for TV (for $p=1$)}
The previous result does not apply directly to the case \( p = 1 \). Nevertheless, cutoff is still expected to hold for in total variation for diffusions. A recent result in \cite{Sal25} shows that the product condition is also sufficient to guarantee cutoff in total variation, provided the diffusion has nonnegative curvature. This condition is satisfied by the Dyson--Jacobi process \eqref{eq:DJ}; see Lemma \ref{lem:cd}.
To verify the product condition, one may rely on \cite[Lemma 2.2]{SC94}
.

\section{Diffeomorphic deformations of diffusions}\label{sec:ddd}

When dealing with a non-constant diffusion coefficient, as in \eqref{eq:IG}, two conceptual approaches naturally arise. One may work directly within the intrinsic geometry of the process. Alternatively, one may adopt a Euclidean viewpoint by performing a change of variables that flattens the diffusion coefficient, thereby absorbing its inhomogeneity and recovering a Euclidean structure. This section is devoted to the deformation of a Markov diffusion with diagonal diffusion via a diffeomorphism, and to the study of which structural properties are preserved under this transformation. The arguments presented here apply to a broader class of diffusion processes.
We recall the necessary background from \cite[Appendix B]{Cha25}.

\paragraph{Diagonal diffusions}
We refer to a diffusion as \emph{diagonal} when the diffusion coefficient has a product structure. Each diffusion coefficient then depends only on the position of the corresponding particle. In the absence of interaction, which is present only through the drift term, the system consists of $n$ independent and identically distributed ergodic particles.

Let $(X_t^n)_{t\ge 0}$ be such a diagonal diffusion, that is, a Markov diffusion process on $\mathbb{R}^n$ solving the stochastic differential equation
\begin{equation}\label{eq:dp}
    \d X_t^{i,n} \= \frac{\sqrt{2}}{a(X_t^{i,n})} \d B_t^i + b_i(X_t^{n}) \d t, \qquad 1 \leq i \leq n,
\end{equation}
where $(B_t)_{t\ge 0}$ is a standard Brownian motion on $\mathbb{R}^n$, $a$ is a smooth, strictly positive function whose inverse $a^{-1}$ is Lipschitz, and each $b_i$ is a $C^2$ vector field that is Lipschitz. Let $A$ be an antiderivative of $a$. The diffusion operator associated with this process is given by
\begin{equation}\label{eq:ddol}
    L_A \= \sum_{i=1}^{n} \frac{\partial^2_{ii}}{a(x_i)^2} + \sum_{i=1}^{n} b^i(x) \partial_i.
\end{equation}
This structure endows the state space with the geometry of a smooth $n$-dimensional manifold, equipped with the Riemannian metric
\begin{equation}\label{eq:mdc}
g_{ij}(x) \= \begin{acc}
	$a(x_i)^2$, & if $i=j$,\\
	$0$, & if $i \neq j$.
\end{acc}
\end{equation}
Noting that $g^{ii}(x) \= a(x_i)^{-2}$, the generator $L_A$ is thus a diffusion operator on this state space.
Moreover, for any smooth function $f \in C^\infty((0,1)^n)$, the associated carré du champ operator is
\begin{equation}
\Gamma_A(f)(x) \= \sum_{i=1}^{n} \frac{(\partial_i f(x))^2}{a(x_i)^2} \= \vabs{\nabla_g f}_g^2.
\end{equation}

\paragraph{Flattening: Euclidean diagonal diffusion}
Since $a$ is strictly positive, its antiderivative $A$ is strictly increasing and defines a $C^{2}$-diffeomorphism onto its image. For each $1 \leq i \leq n$, define $A_i : x \in \RR^n \lms A(x_i)$, and set $A^{n} := (A_1, \cdots, A_n)$ and $A^{-n} := (A_1^{-1}, \cdots, A_n^{-1})$. By Itô's formula, the transformed diffusion process
\begin{equation}
(Y_t^n)_{t\ge 0} = (A^{n}(X_t^n))_{t\ge 0},
\end{equation}
with initial condition $Y_0^n = A^{n}(X_0^n)$, satisfies the stochastic differential equation
\begin{equation}\label{eq:dp}
    \d Y_t^{i,n} \= \sqrt{2}\, \d B_t^i + \parent{a(X_t^{i,n}) \, b_i (X_t^{n}) + \frac{a' (X_t^{i,n})}{a^2(X_t^{i,n})}}  \d t.
\end{equation}
The diffusion operator associated with this Euclidean process is given by
\begin{equation}\label{eq:ddol}
	L \= \sum_{i=1}^{n} \partial^2_{ii} + \sum_{i=1}^{n} \parent{(a \circ A^{-1})(b^i \circ A^{-n}) + \frac{a' \circ A^{-1}}{a^2 \circ A^{-1}}} \partial_i.
\end{equation}
Similarly, for any smooth function $f \in C^\infty$, the corresponding carré du champ operator is
\begin{equation}
\Gamma(f)(x) \= \vabs{\nabla f}^2.
\end{equation}

\paragraph{Stability of properties}
The following lemma is central to our analysis. It asserts that the class of diagonal diffusion processes is globally stable under diffeomorphisms. These transformations preserve the spectral structure (inducing an $\L^2$ isometry) and also transfer the curvature-dimension bounds exactly.

\begin{lem}\label{lem:ddd}
Under the assumptions of this section, consider the following linear mapping and its inverse
\[\Phi : f \lms f\circ A^{n} \and \Phi^{-1} : f \lms f\circ A^{-n}.\]
\begin{itemnum}{\rom}
	\item generator and semigroup:
	\[L = \Phi^{-1}\circ L_A \circ \Phi \and P_t = \Phi^{-1}\circ P_t^A \circ \Phi.\]
	\item invariant measure: if $Y$ is a gradient diffusion de potentiel $V$, then $X$ admits an invariant measure given by
	\[\d \pi_X \,=\, \exp \parent{- V\circ A^n(x) + \sum_i a(x_i)} \d x.\]
	\item spectral resolution: $L_A$ and $L$ share the same spectral resolution since for $\d \pi_Y := e^{-V(x)} \d x$
	\[\Phi : \L^2(\pi_Y) \lra \L^2(\pi_X) \textm{is an isometry}.\]
	\item curvature-dimension: we have $\Gamma^A \=  \Phi\circ \Gamma \circ \Phi^{-1}$ and $\Gamma_2^A \=  \Phi\circ \Gamma_2 \circ \Phi^{-1}$, so that
	\[X \textm{satisfies } \C\D(\rho, m) \qquad\Llra\qquad Y \textm{satisfies } \C\D(\rho, m).\]
\end{itemnum}
\end{lem}

\begin{proof}
Generator and semigroup. The infinitesimal generators $L_A$ and $L$ satisfy
\[L(f)(y) = L_A(f\circ A^{n})(A^{-n}(y)).\]
The semigroup $(P_t)$ is given for every bounded measurable $f:\mathbb{R}^n\to \mathbb{R}$,
\[P_t(f)(y)= \EE(f(Y_t)\mid Y_0=y)= \EE(f(A^n(X_t))\mid \varphi(X_0)=y)= P_t^n(f\circ A^n)(A^{-n}(y)).\]

Invariant measure. If $Y$ admits an invariant measure $\mu$, then $\mu\circ\varphi^{-1}$ is invariant for $X$. But since $Y$ is a gradient diffusion de potentiel $V$, a invariant measure is given by
\[\mu = e^{- V(x)} \d x.\]

Isometry and spectrum. One can view $G$ as an operator on $L^2(\pi_X)$ and the map $\Phi$ is an isometry from $L^2(\pi_Y)$ to $L^2(\pi_X)$ since $\pi_Y = \pi_X\circ A^{-n}$ and \[\int (\Phi f)^2 \, \d \pi_X \= \int (f\circ A^n)^2 \, \d\pi_X \= \int f^2 \, \d\pi_Y.\]
Thus $L_A$ and $L$ share the same spectral resolution.

Curvature-dimension. Thanks to conjugacy, $X$ and $Y$ satisfy the same curvature-dimension inequalities.
The carré du champ $\Gamma$ and $\Gamma_2$ operators transform via the change of variables formulas
\[\Gamma^g(f) \= \sum_{i} g_{ii} ((\nabla_g f)^i)^2, \quad \Gamma_2^g(f) \= \Ric_g(\nabla_g f, \nabla_g f) + \norme{\Hess_g (f)}{HS}^2 + \Hess_g(\V)(\nabla_g f, \nabla_g f).\]
In this setting, one observes that
\begin{eqnarray*}
\partial_i (f\circ A^{-1})(y_i) &\=& \frac{(\partial_i f)(x_i)}{a(x_i)} \= a(x_i) (\nabla_g f)^i(x_i),\\
\partial_{ii}^2 (f\circ A^{-1})(y_i) &\=& \frac{(\partial_{ii}^2 f)(x_i)}{a(x_i)^2} - \frac{a'(x_i)}{a(x_i)^3} (\partial_i f)(x_i) \= \frac{(\Hess_g f)_{ii}(x_i)}{a(x_i)^2}.\\
\partial_{ij}^2 (f\circ A^{-1})(y_i) &\=& \frac{(\partial_{ij}^2 f)(x_i)}{a(x_i) a(x_j)} \= \frac{(\Hess_g f)_{ij}(x_i)}{a(x_i) a(x_j)}.
\end{eqnarray*}
For the le carré du champ operator, one has
\[\Gamma^g f (x) \= \sum_{i} \parent{a(x_i)(\nabla_g f)^i(x_i)}^2 \= \sum_{i} \parent{\partial_i (f\circ A^{-1})(y_i)}^2 \=  \Gamma(f\circ A^{-n})(y).\]
Concerning the Gamma-two operator, we first note that in the diagonal case $\Ric_g(\nabla_g f, \nabla_g f) = 0$.
Moreover,
\[\norme{\Hess_g (f)}{HS}^2 \= \sum_{i, j} g^{ii} g^{j j} (\Hess_g f)_{ij}^2 \= \sum_{i, j} \frac{(\Hess_g f)_{ij}^2}{a(x_i)^2 a(x_j)^2} \= \sum_{i, j} (\partial_{ij}^2 (f\circ A^{-1}))^2 \= \norme{\Hess (f\circ A^{-1})}{HS}^2.\]
which yields
\begin{eqnarray*}
\Hess_g(\V)(\nabla_g f, \nabla_g f) &\=& \sum_{i,j} (\Hess_g V)_{ij} (\nabla_g^i f) (\nabla_g^j f)\\
	&\=& \sum_{i,j} (\Hess_g V)_{ij} \frac{\partial_i (f\circ A^{-1})(y_i)}{a(x_i)} \frac{\partial_j (f\circ A^{-1})(y_i)}{a(x_j)}\\
	&\=& \sum_{i,j} (\Hess \, V)_{ij} \partial_i (f\circ A^{-1}) \partial_j (f\circ A^{-1})\\
	&\=& (\nabla (f\circ A^{-1}))^\top \Hess(V)\,(\nabla (f\circ A^{-1})).
\end{eqnarray*}
Thus $\Gamma_2^g(f)(x) \= \Gamma_2(f\circ A^{-n})(y)$.
\end{proof}

\paragraph{Distances and divergences}
We now consider comparisons and properties of standard distances and divergences between probability measures, namely the total variation distance, the Kullback--Leibler divergence, the $\L^2$ distance, and the Wasserstein distance. Recall that the Wasserstein distance considered here is the \textit{intrinsic Wasserstein distance} of order $r\geq0$, defined for probability measures $\mu$ and $\nu$ on the same space by
\begin{equation}
\Wasserstein_{g,r}^A(\mu, \nu) \;:=\; 2\parent{\inf_{( X, Y)} \EE\intervalle{\parent{\sum_{i=1}^n \vabs{A(X_i)-A(Y_i)}^2}^{r/2}}}^{1/r},
\end{equation}
where the infimum runs over all couplings $(X, Y)$ of $\mu$ and $\nu$.

\begin{lem}
For $\dist \in \{\T\V, \Kullback, \L^2\}$ we have that for probability measures $\mu$ and $\nu$ on the same space
\[\dist \parent{\mu \circ A \mid \nu \circ A} \= \dist \parent{\mu \mid \nu}.\]
For the Wasserstein distance we have
\[\Wasserstein_{r}^A\parent{\mu \mid \nu} \= 2\,\Wasserstein_{r}\parent{\mu \circ A \mid \nu \circ A}.\]
The two processes therefore have the same long-term behavior and equilibrium trend.
\end{lem}

\begin{proof}
In the first case, the result follows directly from the contraction properties of the these distances and divergences (see, e.g., \cite[Lemma 1.2]{BCL}), since $A$ is a homeomorphism.

For the Wasserstein distance, its definition is inherently geometric and depends on the Lipschitz properties of functions. The construction developed in \cite[Appendix B]{Cha25} implies that the result follows immediately.
\end{proof}

The intrinsic Riemannian distance is therefore naturally suited to the definition of Wasserstein distances, intrinsic regularization estimates, and intrinsic functional inequalities. However, the previous lemma allows us to exploit the extensive collection of functional inequalities available in the Euclidean setting in order to derive classical functional inequalities even for intrinsic distances.
Indeed, in the Euclidean framework, these distances and divergences are related through several functional inequalities: the total variation distance and the Kullback--Leibler divergence are connected via the Pinsker's inequality and its converse, as in \cite{Sal25}; and the Kullback--Leibler divergence can be bounded in terms of the $\L^2$ norm, as shown in \cite{BCL}.

\section{Proofs}\label{sec:proofs}



In this section, we establish cutoff for the intrinsic Wasserstein distance.
We first extract the properties arising from the Riemannian viewpoint of the process. This includes, in particular, the spectral properties of the generator established in \cite{Lasb}. We then flatten the process and study the operators associated with the resulting Euclidean diffusion. In particular, we establish a curvature-dimension inequality for this Euclidean process. By the results of the previous section \ref{sec:ddd}, the same inequality then holds for the original Riemannian process.

As a consequence, the results of \cite[Appendix B]{Cha25} apply. We obtain a lower bound for the intrinsic Wasserstein distance by combining the regularization inequality from \cite[Lemma B.2]{Cha25} with Pinsker's inequality, and by using the lower bound established in \cite[Lemma 2.2]{SC94} for the total variation distance. For the upper bound, we rely on the exponential decay of the intrinsic Wasserstein distance obtained in \cite[Lemma B.4]{Cha25}.

\subsection{Useful properties from the Dyson--Jacobi process}\label{ss:updj}

Several key properties of the Dyson--Jacobi process follow directly from its infinitesimal generator, which enjoys a remarkable spectral structure. Access to explicit eigenfunctions is a major advantage in the analysis of cutoff phenomena. Since both the drift and diffusion coefficients are polynomial functions of the coordinates, it is natural to expect the eigenfunctions to be polynomial as well. Moreover, as the generator is self-adjoint and its eigenfunctions can be chosen orthogonal, one is led naturally to families of orthogonal polynomials. This spectral analysis was carried out by Lassalle \cite{Lasb} (see also \cite{BF97}), who treats the case of general $\beta$ and shows that the eigenfunctions are multivariate Jacobi polynomials. 
Let $\L_{\text{sym}}^2((0,1)^n,\pi_{\beta}^{n})$ denote the Hilbert space of symmetric functions in \( n \) variables \( (x_1, \dots, x_n) \) on \( (0,1)^n \), square-integrable with respect to the invariant measure \( \pi_{\beta}^{n} \).

Using the identity $x_i(1-x_j) + x_j(1-x_i) = 2x_i(1-x_i) + (2 x_i- 1) (x_i - x_j)$, inside the double sum, the infinitesimal generator of the Dyson--Jacobi process can be written as
\begin{equation}\label{eq:IG}
\G_{\D\J}(f) = \sum_{i=1}^{n} x_i(1-x_i)(\partial_{ii}^2 f) + \sum_{i=1}^n \parent{b_n- \frac{\beta}{2}(n-1) + \lambda_n x_i + \beta \sum_{j \neq i} \frac{x_i (1-x_i)}{x_i - x_j}}(\partial_i f),
\end{equation}
where $\lambda_n := a_n+b_n$ denotes the spectral gap.
The operator $\G_{\D\J}$ is self-adjoint on $\L_{\text{sym}}^2(\RR_+^n,\pi_{\beta}^{n})$, acts as a second-order differential operator, and preserves the space of symmetric polynomials. Its spectrum is completely described in \cite{Lasb}: all eigenspaces are finite-dimensional, and the eigenfunctions are given by the generalized Jacobi polynomials. The eigenspace associated with the eigenvalue $-\lambda_n$ is one-dimensional and generated by an explicit affine eigenfunction
\begin{equation}\label{eq:efsp}
\varphi_{n}(x) \;:=\; \phi_n(x) - \frac{n b_n}{b_n + a_n} \qquad\mtext{where}\qquad \phi_n(x) \;:=\; x_1 + \cdots + x_n.
\end{equation}
We emphasize that while the underlying stochastic process is referred to as the \textit{Dyson--Jacobi} process, its infinitesimal generator is called the \textit{Jacobi} operator.

Thanks to the simple form of this eigenfunction, its moments and norms can be computed explicitly. First, since the diffusion coefficient is non-constant, the carré du champ operator is no longer given by the squared Euclidean gradient, but instead takes the deformed form
\begin{equation}\label{eq:cdc}
\Gamma^{\D\J} f \= \sum_{i=1}^n x_i (1-x_i) (\partial_{i} f)^2.
\end{equation}
Since $\nabla \varphi_{n} = (1, \ldots, 1)$, the expectation of this eigenfunction with respect to $\pi_{\beta}^{n}$ is given by
\begin{equation}\label{eq:espvp}
\EE_{\pi_{\beta}^{n}}(\varphi_n) \= - \EE_{\pi_{\beta}^{n}}(1 \times \G\varphi_{n}) \=  \EE_{\pi_{\beta}^{n}}(\Gamma^{\D\J}(1,\varphi_{n})) \= 0.
\end{equation}
Moreover, computing its $\L^2(\pi_{\beta}^{n})$ norm by integration by parts yields
\[\norme{\varphi_{n}}{2}^2 \= - \ps{\G_{\D\J}\varphi_{n}, \varphi_{n}}_{2} \= \int \sum_i x_i (\partial_i \varphi_{n})^2 \, \d \pi_{\beta}^{n} \= \EE_{\pi_{\beta}^{n}}(\varphi_n) + \frac{n b_n}{b_n + a_n} \= \frac{n b_n}{b_n + a_n}.\]
These explicit computations play a central role in the study of cutoff in $\L^p$ distances, as discussed for instance in Section \ref{ss:krc}.

Finally, we note that establishing curvature-dimension properties directly in the intrinsic Riemannian framework is delicate in this setting. A context in which the computations become more transparent is obtained by flattening the diffusion via a suitable change of variables. This approach is developed in the following section.

\subsection{Useful properties of the Euclidean--Dyson--Jacobi process} \label{sec:epov}

We begin by describing the Euclidean--Dyson--Jacobi process.
Using the construction of Section \ref{sec:ddd}, one can eliminate the non-constant diffusion coefficient and work entirely within a Euclidean framework. For the Dyson--Jacobi process \eqref{eq:DJ}, this is achieved through the change of variables
\[Y_{t}^{i,n} \;:=\; 2 \arcsin \sqrt{ X_{t}^{i,n}},\]
which leads to what is referred to as the \textit{Euclidean--Dyson--Jacobi process} (EDJ).
After this transformation, the resulting stochastic differential equation takes the form
\begin{equation}\label{eq:EDJ}
\d Y_{t}^{n} \= \sqrt{2}\, \d B_t - \nabla V(Y_t^{n}) \, \d t, \qquad Y_0^{i,n} = 2 \arcsin \sqrt{ X_{0}^{i,n}},
\end{equation}
where the potential function is given explicitly in \cite[Equation (2)]{Dem10} by
\[y \mapsto -\sum_{i=1}^n 2 (b_n-a_n) \ln \vabs{\sin y_i/2} \,+ C_a \ln \vabs{\sin y_i} +\, 2\beta \sum_{i > j} \ln\vabs{\sin\parent{\frac{y_i-y_j}{2}}} + \ln\vabs{\sin\parent{\frac{y_i+y_j}{2}}}.\]
Up to an additive constant, which plays no role in the dynamics, we may assume without loss of generality that
\begin{equation}
V(y) \;:=\; -\sum_{i=1}^n C_b \ln \vabs{\sin \frac{y_i}{2}} \,+\, C_a \ln \vabs{\cos \frac{y_i}{2}} +\, 2 \beta \sum_{i > j} \ln\vabs{\sin\parent{\frac{y_i-y_j}{2}}} + \ln\vabs{\sin\parent{\frac{y_i+y_j}{2}}},
\end{equation}
with $C_a = 2 a_n - \beta (n-1)-1$ and $C_b = 2 b_n - \beta (n-1) -1$. Using \eqref{eq:exconv}, the potential $V$ is thus strictly convex, of class $C^2$.


The unique invariant probability measure is the Gibbs measure $\d P_\beta^n := e^{-V(y)} \d y$, namely
\begin{equation}\d P_\beta^n \= \frac{\1_{(y_1, \cdots, y_n) \in \bar{D}_n}}{C_\beta^n} \prod_{i=1}^n \vabs{\sin \frac{y_i}{2}}^{C_b} \vabs{\cos \frac{y_i}{2}}^{C_a} \prod_{j < i} 
\vabs{\sin\parent{\frac{y_i-y_j}{2}} \sin\parent{\frac{y_i+y_j}{2}}}^{2\beta} \d y_i.
\end{equation}
This measure is reversible for the EDJ dynamics, and the diffeomorphic transformation therefore preserves the log-concavity of the invariant distribution.

The associated infinitesimal generator is the linear differential operator
\begin{equation}
\G_{\E\D\L} \= \Delta - \nabla V \cdot \nabla,
\end{equation}
acting on smooth functions. Combining the results of Section \ref{sec:ddd} and Section \ref{ss:updj}, we obtain a complete spectral description of this operator. In particular, the eigenfunction associated with the spectral gap $\lambda_n \= b_n + a_n$ is given by
\[\widetilde{\varphi}_n(y) \= \sum_{i=1}^n \sin^2(y_i/2) + \frac{n b_n}{b_n + a_n}.\]

The key advantage of this flattening procedure lies in the curvature-dimension structure, and consequently in the analysis of cutoff. In the Euclidean setting, the carré du champ operator takes its classical form
\begin{equation}\label{eq:gammaEDJ}
\Gamma^{\E\D\J}(f) \= \vabs{\nabla f}^2 \= \sum_{i=1}^{n} (\partial_{i} f)^2.
\end{equation}
For the computation of the Gamma-two operator, recall that in the Euclidean framework it is given, in terms of the potential, by $\Gamma_{2}^{\E\D\L}(f) \= \norme{\Hess f}{2}^2 + (\nabla f)^\top (\Hess \, \V) (\nabla f)$.
The entries of the Hessian of the potential are explicitly given by
\[\Hess(V)_{ii} \= \frac{C_b}{4 \sin(y_i)^2} + \frac{C_a}{4 \cos(y_i)^2} + \frac{\beta}{2} \sum_{k\neq i} \frac{1}{\sin(y_i+y_k)^2} + \frac{1}{\sin(y_i-y_k)^2},\]
and
\[\Hess(V)_{ij} \= \frac{\beta}{2} \parent{\frac{1}{\sin(y_i+y_j)^2} + \frac{1}{\sin(y_i-y_j)^2}}.\]
Consequently, the Gamma-two operator can be written as
\begin{align}\label{eq:gamma2EDJ}
\Gamma_{2}^{\E\D\J}(f) &\= \norme{\Hess f}{2}^2 + \frac{1}{4} \sum_{i} \parent{\frac{C_b}{\sin(y_i)^2} + \frac{C_a}{\cos(y_i)^2}}(\partial_i f)^2\\
	& \qquad + \; \frac{\beta}{2} \sum_{i, k < i} \parent{\frac{1}{\sin(y_i+y_j)^2} + \frac{1}{\sin(y_i-y_j)^2}} \parent{(\partial_i f) + (\partial_j f)}^2,
\end{align}
after symmetrizing the double sums using
\begin{align*}
	& \frac{\beta}{2} \sum_{i, j\neq i} \parent{\frac{1}{\sin(y_i+y_j)^2} + \frac{1}{\sin(y_i-y_j)^2}} [(\partial_i f)^2 + (\partial_i f)(\partial_j f)]\\
	& \qquad \= \frac{\beta}{2} \sum_{i, j < i} \parent{\frac{1}{\sin(y_i+y_j)^2} + \frac{1}{\sin(y_i-y_j)^2}} [(\partial_i f)^2 + (\partial_j f)^2 + 2 (\partial_i f)(\partial_j f)].
\end{align*}
Using the expressions \eqref{eq:gammaEDJ} and \eqref{eq:gammaEDJ}, together with the bounds $\frac{1}{\sin(y_i)^2}, \frac{1}{\cos(y_i)^2} \geq 1$, we obtain
\[\Gamma_{2}^{\E\D\J}(f) \geq \frac{C_a + C_b}{4} \Gamma(f),\]
and hence conclude that Euclidean--Dyson--Jacobi process satisfies the curvature-dimension condition $\C\D\parent{\frac{1}{4}\parent{C_a + C_b},\infty}$. In particular, applying the results of Section \ref{sec:ddd}, we obtain the following
\begin{lem}[Curvature-dimension]\label{lem:cd}
The Dyson--Jacobi process \eqref{eq:DJ} satisfies the curvature-dimension condition $\C\D\parent{\frac{1}{2}\parent{a_n  + b_n - \beta (n-1)-1},\infty}$.
\end{lem}

Finally, we note that the non-interacting case $(\beta=0)$ was already treated in \cite[Section 9]{SC94} and \cite[Section 2.7.4]{BGL14}. In this way, we recover instrinsically a full set of functional inequalities. In particular, under the product condition, the cutoff phenomenon in the total variation established in \cite{Sal25} applies in the present setting.

\subsection{Proof of Theorem \ref{thm:CDJ}}

Recall that the spectral gap is given by $\lambda_n := a_n + b_n$. The associated eigenfunction $\varphi_n$ is identified in Section \ref{ss:updj}. We also combine Lemma \ref{lem:cd} with the fact that, under assumptions \eqref{eq:exconv}, we have the following curvature bound
\begin{equation}\label{eq:csp}
\fa n \geq 1, \qquad 2 \rho_n \;:=\; a_n + b_n - \beta (n-1) - 1 \> 1 .
\end{equation}

\paragraph{Lowerbound}
From Pinsker's inequality and the regularization inequality for the intrinsic Wasserstein distance obtained in \cite[Lemma B.2]{Cha25}, we get, for each $\eta >0$
\[\norme{\mu_{t+\eta}^{x} - \pi_\beta^n}{\T\V}^2 \leq 2 \cdot \Kullback(\mu_{t+\eta}^{x} \mid \pi_\beta^n) \leq \frac{1}{2 \eta} \cdot \Wasserstein^2(\mu_t^x, \pi_{\beta}^{n}).\]
We then control this lower bound using \cite[Lemma 2.2]{SC94},
\begin{equation}\label{eq:lbtv}
\norme{\mu_t^x - \pi_\beta^n}{\T\V} \geq 1 - 4 \frac{\norme{\varphi_n}{2}^2}{\vabs{\varphi_n(x)}^2} e^{2 \lambda_n t} - 4 \frac{\norme{\varphi_n}{2}^4}{\vabs{\varphi_n(x)}^4} \Var_{\mu_t^x}\parent{\frac{\varphi_n(x) \varphi_n(\cdot)}{\norme{\varphi_n}{2}^2} } e^{2 \lambda_n t}.
\end{equation}
By Duhamel's formula (see, for instance, \cite[Equation (3.1.21)]{BGL14}) and the identity $\P_t \varphi_n = e^{- \lambda_n t} \varphi_n$, we obtain
\[\Var_{\mu_t^x}(\varphi_n) \= \P_t(\varphi_n^2)(x) - (\P_t \varphi_n(x))^2 \= 2 \int_0^t \P_s (\Gamma (\P_{t-s}\varphi_n))(x) \d s 
\leq \frac{n}{2 \lambda_n},\]
since
\[\P_s (\Gamma (\P_{t-s}\varphi_n))(x) \= e^{-\lambda_n(t-s)} \P_s \Gamma \varphi_n \= e^{-\lambda_n(t-s)} \P_s \parent{\sum_{i=1}^n x_i(1-x_i)} \leq \frac{n}{4} e^{-\lambda_n(t-s)}.\]
Therefore, since $b_n \geq 1$, it follows that
\[\norme{\mu_t^x - \pi_\beta^n}{\T\V} \geq 
1 - 6 \, \frac{n b_n}{\lambda_n \vabs{\varphi_n(x)}^2} \, e^{2 \lambda_n t}.\]

\paragraph{Upperbound}
Denoting $\arcsin\sqrt{x_0^n} := \parent{\arcsin\sqrt{x_0^{n,i}}}_{1\leq i \leq n}$, we estimate
\[\Wasserstein_{}^2(\delta_{x^n_0}, \pi_{\beta}^{n}) \= \int\vabs{\arcsin\sqrt{x^n_0} - \arcsin\sqrt{x}}^2 \d \pi_{\beta}^{n}(x).\]
Since $\vabs{a-b}^2 \leq (\vabs{a}^2 + \vabs{b}^2)$, and since $\arcsin \sqrt{t} \leq \frac{\pi}{2} \sqrt{t}$ for all $0\leq t \leq 1$, we derive
\[\vabs{\arcsin\sqrt{x^n_0} - \arcsin\sqrt{x}}^2 \leq \parent{\vabs{\arcsin\sqrt{x^n_0}}^2 + \vabs{\arcsin\sqrt{x}}^2}
\leq \frac{\pi^2}{4} \parent{\phi_n(x^n_0) + \phi_n(x)}.\]
Using $b_n \leq \lambda_n$, together with $\EE_{\pi_{\beta}^{n}}(\phi_n)  \= \frac{n b_n}{\lambda_n}$, it then follows
\begin{eqnarray*}
\Wasserstein^2(\Law(X_t^n), \pi_{\beta}^{n}) & \= & e^{- 2 \rho_n t} \cdot \Wasserstein^2(\Law( X_0^n), \pi_{\beta}^{n})\\
	&\leq& e^{-2 \rho_n t} \cdot \frac{\pi^2}{4} \parent{\phi_n(x^n_0) + \EE_{\pi_{\beta}^{n}}(\phi_n)}\\
	&\leq& \frac{\pi^2}{4} e^{-2 \rho_n t} \parent{\phi_n(x^n_0) + \frac{n b_n}{\lambda_n}}.
\end{eqnarray*}

\subsection{Proof of Corollary \ref{cor:CDJ}}
This corollary allows us to establish cutoff in several regimes indexed by the ratio $r\in[0,+\infty]$. We proceed by revisiting the previously derived lower bounds $c_n^{-}$ and upper bounds $c_n^{+}$. We then study the ratio $c_n^{-}/c_n^{+}$, keeping in mind that for all $n\geq 1$ we have $c_n^{-} \leq c_n \leq c_n^{+}$.

\paragraph{Case $r=+\infty$}
This assumption implies the following asymptotic relations
\[a_n \gg b_n, \qquad a_n \gg \frac{\beta}{2} n, \qquad \frac{\rho_n}{\lambda_n} \xra{}{n\to\infty} \frac{1}{2}.\]
In this regime, the only meaningful situation occurs when $\log(a_n/b_n) \lesssim \log n$. In particular, if the ratio $a_n / b_n$ grows faster than any polynomial, the lower bound becomes negligible compared to the upper bound. More precisely, we have $c_n^{-} \ll c_n^{+}$, and the resulting estimates of the mixing time $c_n$ is very poor.

\paragraph{Case $r=0$}
We first note that one must assume $\bar{x}_0 \neq 1$. Otherwise, the lower bound given by equation \eqref{eq:lbtv} does not allow us to conclude, and we only obtain an upper bound of order $\frac{\log n}{\lambda_n}$. We therefore assume $\bar{x}_0 \neq 1$. Under the assumption $r=0$, we further have
\[a_n \ll b_n, \qquad b_n \gg \frac{\beta}{2} n, \qquad \lambda_n \sim b_n, \qquad \frac{\rho_n}{\lambda_n} \xra{}{n\to\infty} \frac{1}{2}.\]

\paragraph{Case $0 \< r \< +\infty$}
The same preliminary remark applies as in the previous case. We must assume $\bar{x}_0 \neq m$ for the same reason. Otherwise, we only obtain an upper bound of order $\frac{\log n}{2 \rho_n}$. As a consequence of the assumption $0 \< r \< 1$, we have
\[a_n \sim r b_n, \qquad \lambda_n \sim (r+1)b_n, \qquad 1-\frac{2\rho_n}{\lambda_n} \sim \frac{\beta n}{\lambda_n} \sim \frac{\beta}{r+1} \frac{n}{b_n}.\]
The key quantity in this regime is the ratio $\frac{\rho_n}{\lambda_n}$. Its behavior is determined by the growth of the coefficient $b_n$. 

If $ \rho_n \ll \lambda_n$, then the estimate of the mixing time is poor, since $c_n^{-} \ll c_n^{+}$.

Otherwise, if $\rho_n \sim \alpha \lambda_n$ for some $0 \< \alpha \< \frac{1}{2}$, the conclusion of the corollary follows.

\bibliography{bibliographie_jacobi}
\bibliographystyle{abbrv}

\end{document}